\documentclass[twoside,10pt,a4paper,intlimits]{amsart}

\usepackage[latin1]{inputenc}

\usepackage[centertags]{amsmath}

\usepackage{amsmath}
\usepackage{amsthm}
\usepackage{amssymb}
\usepackage{amsfonts}
\usepackage{amsxtra}
\usepackage{bbm}
\usepackage{color}
\usepackage{xspace}
\usepackage{mathabx}
\makeatletter
\let\diameter\@undefined                        
\makeatother

\usepackage{enumitem}

\usepackage{ifthen}
\usepackage{graphicx}

\newtheorem{theorem}{Theorem}[section]

\newtheorem{corollary}[theorem]{Corollary}

\newtheorem{remark}[theorem]{Remark}

\definecolor{lightgrey}{rgb}{.7,.7,.7}

\newcommand{\hellgrau}[1]{\textcolor{lightgrey}{#1}}

\newcounter{notectr}
\newcommand{\note}[1]{\ifthenelse{\thenotectr=1}{\hellgrau{#1}}{}}

\newcommand{\argmin}{\ensuremath{\operatorname{argmin}}}

\newcommand{\divo}{\operatorname{div}}

\numberwithin{equation}{section}

\newcommand{\definedas}{\mathrel{:=}}
\newcommand{\energy}{\ensuremath\mathcal{J}}

\providecommand{\setN}{\ensuremath{\mathbb{N}}}
\providecommand{\setR}{\ensuremath{\mathbb{R}}}

\newcommand{\ie}{i.e.,\xspace}

\newcommand{\dx}{\,{\rm d}x}

\providecommand{\abstmp}[2]{{#1\lvert{#2}#1\rvert}}
\providecommand{\abs}[1]{\abstmp{}{#1}}
\providecommand{\normmtmp}[2]{{#1\lvert\hspace{-0.07em}#1\lvert\hspace{-0.07em}#1\lvert{#2}
    #1\rvert\hspace{-0.07em}#1\rvert\hspace{-0.07em}#1\rvert}}  
\providecommand{\normm}[1]{\normmtmp{}{#1}}

\providecommand{\normtmp}[2]{{#1\lVert{#2}#1\rVert}}
\providecommand{\norm}[1]{\normtmp{}{#1}}

\begin{document}

\title[Enforcing a Discrete Maximum Principle by a Cutoff]{A Note on why Enforcing
  Discrete Maximum Principles by a simple a Posteriori Cutoff\\ is a Good Idea}

\author[C.~Kreuzer]{Christian Kreuzer}
\address{Christian Kreuzer,
 Fakult{\"a}t f{\"u}r Mathematik,
 Ruhr-Universit{\"a}t Bochum,
 Universit{\"a}tsstrasse 150, D-44801 Bochum, Germany
 }%
\urladdr{http://www.ruhr-uni-bochum.de/ffm/Lehrstuehle/Kreuzer/}
\email{christan.kreuzer@rub.de}

\begin{abstract}
  Discrete maximum principles in the approximation of partial
  differential equations are crucial for the preservation of
  qualitative properties of physical models. 
  In this work we enforce the discrete maximum principle by performing a simple
  cutoff. We show that for many problems this 
  a posteriori procedure even improves the approximation in the
  natural energy norm. 
  The results apply to many different kinds of approximations
  including conforming higher order and $hp$-finite elements. Moreover
  in the case of finite element approximations there is no geometrical
  restriction
  on the partition of the domain.  
\end{abstract}

\keywords{discrete maximum principle \and finite elements \and nonlinear pde
 \and p-Laplace \and conforming approximation \and reaction diffusion}
  \subjclass[2010]{65N30 \and 35J92 \and 35J47 \and 35J15}

\maketitle


\section{Introduction}
\label{sec:introduction}

Consider a function $u:\Omega\to \setR$ on some bounded domain
$\Omega\subset\setR^d$ such that 
\begin{align}\label{eq:rd}
  -\Delta u + c \,u \le 0
\end{align}
in the variational sense for some nonnegative $c$. Then for $u^+=\max\{0,u\}$, we have the estimate 
\begin{align}\label{eq:mp}
  \sup_{\Omega} u \le \sup_{\partial\Omega} u^+,  
\end{align}
which is well-known as weak maximum principle; compare e.g. with
\cite{GilbargTrudinger:83}. 

Maximum principles usually reflect fundamental physical
principles like e.g. the positivity of the density and it is 
 desirable that numerical schemes respect physical
 principles. Therefore, for a conforming approximation $U$,
computed by some numerical scheme, there arises the question if it
satisfies a so called discrete maximum principle
\begin{align}
  \label{eq:DMP}\tag{DMP}
  \sup_{\Omega} U \le \sup_{\partial\Omega} U^+.
\end{align}

Suppose for example that the approximation $U$ is generated by a 
conforming finite element method. 
In this case there are plenty of results on discrete
maximum principles; without attempting to
provide a complete list, we refer to
\cite{CiarletRaviart:73,Santos:82,
BrandtsKorotovKrizek:2008,LiHuang:10,
DieningKreuzerSchwarzacher:12}. All those results have in common that
they base on piece wise affine finite element spaces and on
strong geometrically restrictions on the underlying partitions of the
domain. To be more precise, they are basically restricted to
non-obtuse or even acute simplicial meshes. Although there are some
results on non-obtuse refinement in
\cite{KorotovKrizek:2011,KorotovKrizek:2005},
it is clear that those restrictions introduce serious complications
for the refinement and meshing of the domain $\Omega$. 
The situation is even less satisfying for 
higher order finite elements let alone $hp$-finite elements. To our
best knowledge for these schemes discrete maximum principles 
are known only  in a relaxed  sense \cite{Schatz:80} or in very
restrictive situations; see
\cite{HoehnMittelmann:81}. 

Based on an analysis of discrete Green's functions, Dr{\u{a}}g{\u{a}}nescu,  Dupont,  and Scott
\cite{DragDupScott:04} suggest that the discrete maximum principle 
for piece wise affine functions may only fail in a small region close
to the boundary. This motivates to simply 
cutoff those small regions with unphysical behavior. In other 
words, we define 
\begin{align}\label{eq:cut}
  U^* = \min\big\{U,\,\sup_{\partial\Omega} U^+\big\}\qquad\text{in}~\Omega.
\end{align}
Obviously the modified function
$U^*$ satisfies the \eqref{eq:DMP} independent of the polynomial degree or some
underlying partition of $\Omega$. Although so far 
this technique has lacked of mathematical justification,
it is actually quite popular in engineering.
   
In this paper we overcome this drawback. In fact, using variational techniques from
\cite{DieningKreuzerSchwarzacher:12,OttLM98,BilFuc02}, we prove in
section~\ref{sec:cutting} that the modified function $U^*$ is even a
better approximation than $U$. To be more precise, if
$u_{|\partial\Omega}=U_{|\partial\Omega}$, then we have 
\begin{align}
  \label{eq:main}
  \normm{u-U^*}\le\normm{u-U}
\end{align}
in the norm
$\normm{\cdot}^2=\int_\Omega\abs{\nabla\cdot}^2+c\,\abs{\cdot}^2\,{\rm
  d}x$, induced by the reaction diffusion differential operator in
\eqref{eq:rd}. In section~\ref{sec:applications} we generalize the
techniques to systems of pde's and the $p$-Laplace operator. 

We emphasize that the results are not restricted to finite element
approximation nor to Galerkin schemes. In fact, the results can be
applied to any conforming approximation of $u$.

\section{Enforcing the Discrete Maximum Principle}
\label{sec:cutting}
In this section we shall prove the main result~\eqref{eq:cut} for functions
 satisfying \eqref{eq:rd}. 
To this end, for the bounded domain $\Omega\subset\setR^d$, $d\in\setN$, let 
$L^2(\Omega)$ and $H^1(\Omega)$ be
the space of square integrable Lebesgue and
Sobolev functions on $\Omega$, respectively. 
We denote by $\big(H^1(\Omega)\big)^*$ the dual of $H^1(\Omega)$ 
and by $H_0^1(\Omega)$ the subspace of functions in $H^1(\Omega)$
with vanishing trace on the boundary $\partial\Omega$.

The variational formulation of \eqref{eq:rd} reads as follows:
Let $0\le c\in L^\infty(\Omega)$, \ie a nonnegative essentially
bounded function.  We assume that $u\in
H^1(\Omega)$ such that  
\begin{align}\label{eq:rdweak}
  \int_\Omega\nabla u\cdot\nabla v +c\, uv\dx  =: F(v)\le 0 \qquad
  \text{for all}~v\in H_0^1(\Omega)~\text{with}~v\ge 0~\text{in}~\Omega.
\end{align}
With this definition we have $F\in \big(H^{1}(\Omega)\big)^*$ 
and it is well known that $u$ satisfies the weak maximum principle
\eqref{eq:mp}; see \cite{GilbargTrudinger:83}.
Moreover, $u$ is the unique minimizer of 
\begin{align}\label{eq:energy}
  \energy(v):=\frac12 \int_\Omega |\nabla v|^2+ c\,|v|^2\dx - F(v) 
\end{align}
in $u+H^1_0(\Omega)$. In other words, $u$ is minimal among all
functions in $H^1(\Omega)$ that coincide with $u$ on the boundary
$\partial\Omega$. 

Let now $U\in H^1(\Omega)$ be some approximation to $u$. Note that 
there is no restriction on the kind of approximation beyond of that it
is conforming; we will
come back to this issue in Remark \ref{r:bnd} and the Conclusion \S\ref{sec:conclusion} below.
It follows by
standard arguments that 
\begin{align}\label{eq:U^*}
  U^*\definedas \min\big\{U,\,\sup_{\partial\Omega} U^+\big\}=
  (U-\sup_{\partial\Omega} U^+\big)^-+\sup_{\partial\Omega} U^+\in H^1(\Omega).
\end{align}
Consequently, $U^*$ satisfies the 
\eqref{eq:DMP}. If $U_{|\partial\Omega}=u_{|\partial\Omega}$, we have
by \eqref{eq:mp} that 
\begin{align*}
  |u(x)-U^*(x)| &= \left\{
    \begin{alignedat}{2}
      &|u(x) - U(x)|, &\quad&\text{if}~ U(x)\le
      \sup_{\partial\Omega} U^+\\
      &
      \sup_{\partial\Omega} U^+ -
      u(x), &\quad&\text{if}~U(x)>
      \sup_{\partial\Omega} U^+
    \end{alignedat}\right\}
    \le \abs{u(x)-U(x)},
\end{align*}
for almost every $x\in\Omega$. Hence, it follows 
\begin{align}\label{eq:L2}
  \norm{u-U^*}_{L^2(\Omega)}\le 
  \norm{u-U}_{L^2(\Omega)}.
\end{align}
Here $\|\cdot\|_{L^2(\Omega)}^2\definedas\int_\Omega|\cdot|^2\dx$
denotes the standard norm  on $L^2(\Omega)$. 

The corresponding estimate \eqref{eq:main} for the energy norm is
less obvious. To prove this estimate we need  
some more properties of the truncated function $U^*$.  
Similarly as in \cite{OttLM98} we have on the one hand, that 
\begin{align*}
  U^*(x)=
  \begin{cases}
    U(x) \qquad&\text{if}~ U(x)\le \sup_{\partial\Omega} U^+,
    \\
    \sup_{\partial\Omega} U^+\ge0,\qquad&\text{if}~ U(x)> \sup_{\partial\Omega} U^+,
  \end{cases}\qquad x\in\Omega.
\end{align*}
and hence we obtain 
\begin{subequations}\label{eq:U*props}
  \begin{align}
    \label{eq:U^*<U}
    U^*\le U\qquad\text{and}\qquad\abs{U^*}&\leq
    \abs{U}\qquad\text{in}~\Omega.
  \end{align}
  On the other hand we have
  \begin{align*}
    \nabla U^*(x) =
    \begin{cases}
      \nabla U(x),\qquad&\text{if}~ U(x)\le \sup_{\partial\Omega} U^+,
      \\
      0,\qquad&\text{if}~ U(x)> \sup_{\partial\Omega} U^+,
    \end{cases} \qquad x\in\Omega
  \end{align*}
  and consequently
  \begin{align}
    \label{eq:nablaU^*<nablaU}
    \abs{\nabla U^*}\le \abs{\nabla U}\qquad\text{in}~\Omega.
  \end{align}
\end{subequations}
These observations are the key properties for proving the following result.
\begin{theorem}\label{thm:main}
  Suppose the conditions of this section. In particular,
  let $U\in H^1(\Omega)$ and let $U^*\in H^1(\Omega)$ be the
  modification of $U$ according to \eqref{eq:U^*}. Then we have for
  the energy $\mathcal{J}\colon H^1(\Omega)\to \setR$ defined in \eqref{eq:energy}, that
  \begin{align}\label{eq:reduction}
    \energy(U^*)\le \energy(U)\qquad\text{and}\qquad
    U^*-U\in H_0^{1}(\Omega).
  \end{align}
\end{theorem}
\begin{proof}
  The second claim is a direct consequence of the definition
  \eqref{eq:U^*} of $U^*$. For the first claim we observe from
  \eqref{eq:nablaU^*<nablaU} that 
  \begin{align*}
    \int_\Omega\abs{\nabla U^*}^2\dx \le     \int_\Omega\abs{\nabla U}^2\dx.
  \end{align*}
  Furthermore, it follows from the second inequality in \eqref{eq:U^*<U} that 
  \begin{align}\label{ineq:L2}
    \int_\Omega \abs{U^*}^2\dx\le\int_\Omega \abs{U}^2\dx
  \end{align}
  and since $F(v)\le 0$ for all $v\in H^1(\Omega)$, $v\ge 0$, we
  obtain 
  \begin{align*}
    F(U)-F(U^*)=F(U-U^*) \le 0\quad\Rightarrow\quad -F(U^*)\le -F(U).
  \end{align*}
  Combining these estimates together with the definition
  \eqref{eq:energy} of the energy
 proves the theorem.
\end{proof}

\begin{corollary}\label{cor:main}
  Suppose the conditions of this sections and assume further that 
  \begin{align*}
    U-u\in H_0^1(\Omega), \qquad\text{\ie}~u=U~\text{on}~\partial\Omega.
  \end{align*}
  Then we have 
  \begin{align*}
    \normm{u-U^*}\le \normm{u-U}
  \end{align*}
  for the energy norm $\normm{\cdot}^2\definedas
  \int_\Omega\abs{\nabla\cdot}^2 +c\,\abs{\cdot}^2\dx$ on
  $H^1_0(\Omega)$ induced by \eqref{eq:rdweak}. 
\end{corollary}

\begin{proof}
  It follows
  from the assumption $u-U\in H_0^1(\Omega)$ that  $U,U^*\in
  u+H_0^1(\Omega)$.  We recall that $u$ is the unique minimizer of the
  energy $\energy$  
  in $u+H_0^1(\Omega)$. Therefore, together with Theorem~\ref{thm:main}, we have  
  \begin{align}\label{eq:4}
    0\le \energy(U^*)-\energy(u)\le \energy(U)-\energy(u).
  \end{align}
  On the other hand, for arbitrary $v\in u+H_0^1(\Omega)$ we have
  $u-v\in H_0^1(\Omega)$ and it follows from
  \eqref{eq:rdweak} that 
  \begin{align*}
    \energy(v)&-\energy(u)= \frac12 \int_\Omega |\nabla v|^2+
    c\,|v|^2\dx  -\frac12 \int_\Omega |\nabla u|^2+ c\,|v|^2\dx
    + F(u-v) 
    \\
    &= \frac12 \int_\Omega |\nabla v|^2+ c\,|v|^2\dx - \int_\Omega
    \nabla v\cdot\nabla u+ c\,vu\dx + \frac12 \int_\Omega |\nabla
    u|^2+ c\,|u|^2\dx 
    \\
    &=\frac12\normm{u-v}^2.
  \end{align*}
  Using this observation with $v=U$ respectively $v=U^*$ in \eqref{eq:4} proves
  the claim. 
\end{proof}
\begin{remark}
  For $c\equiv 0$ the maximum principle \eqref{eq:mp} reads
  as
  \begin{align*}
    \sup_{\Omega}u\le \sup_{\partial \Omega}u.
  \end{align*}
  Consequently, in this case, we define 
  \begin{align*}
    U^*\definedas \min\big\{U,\,\sup_{\partial\Omega} U\big\}.
  \end{align*}
  Since $c\equiv0$ we do not need to
  have the second estimate in \eqref{eq:U^*<U} in order to prove
  Theorem \ref{thm:main}. All other estimates, namely the first
  estimate in \eqref{eq:U^*<U} and \eqref{eq:nablaU^*<nablaU}, stay
  true for $U^*$ as defined above. Therefore, 
  Theorem \ref{thm:main} and Corollary \ref{cor:main} are
  still valid in this case; compare also with the examples of Section~\ref{sec:applications} 
  below.
\end{remark}

\begin{remark}\label{r:bnd}
  In the approximation of solutions to partial differential
  equations  with finite elements one usually
  considers the error of the boundary values and the 
  residual interior $\Omega$, separately; see e.g. \cite{BrennerScott:08}. To be more precise, let
  $u$ satisfy \eqref{eq:rdweak} and let $G$ be a discrete
  approximation of $u$ on $\partial\Omega$. We consider an approximation $U\in
  H_0^1(\Omega)$ with $U=G$ on $\partial\Omega $ to the weak solution $u_G\in H_0^1(\Omega)$ of the problem  
  \begin{align*}
    -\Delta u_G +c\,u_G = -\Delta u +c\,u \quad\text{in}~\Omega,\qquad\text{and}\qquad
    u_G=G~\text{on}~\partial\Omega.
  \end{align*}
  Then $U$ satisfies Theorem
  \ref{thm:main} and Corollary \ref{cor:main} with $u_G$ instead of
  $u$, \ie let $U^*\in H^1(\Omega)$ be defined as in \eqref{eq:U^*}, then
  \begin{align*}
    \normm{u-U^*}\le \normm{u_G-U^*} + \normm{u-u_G}\le \normm{u_G-U} + \normm{u-u_G}
  \end{align*}
  The error $\normm{u-u_G}$ can be estimated by means of the trace $u-G$
  on~$\partial\Omega$. Similar
  techniques can be applied in the case of a curved boundary.
\end{remark}

\section{Extensions}
\label{sec:applications}
In the previous
section, for the ease of presentation,  we restricted ourselves to linear scalar valued 
reaction diffusion problems. However, the presented
ideas can be generalized to more complicated problems. We
shall present two examples.

\subsection{Convex Hull Property}
\label{sec:convex-hull-property}
In this section we shall generalize the results of Section
\ref{sec:cutting} to vector valued functions.
In order to do so, we first have to
generalize the cutoff process in \eqref{eq:cut} to higher
dimensions. To this end, let $K\subset \setR^m$, $m\in\setN$, be a
convex set and define $\Pi_K:\setR^m\to K$ to be the closest point projection
with respect to the Euclidean norm $|\cdot|\colon\setR^m\to\setR$. In other
words
\newcommand{\bsx}{\ensuremath{\boldsymbol{x}}}
\newcommand{\bsy}{\ensuremath{\boldsymbol{y}}}
\newcommand{\bsu}{\ensuremath{\boldsymbol{u}}}
\newcommand{\bsU}{\ensuremath{\boldsymbol{U}}}
\newcommand{\bsv}{\ensuremath{\boldsymbol{v}}}
\begin{align}\label{eq:PiK}
  \Pi_K x := \argmin_{y \in K} \abs{x-y}.
\end{align} 
This definition can be
extended to 
\begin{align*}
  \Pi_K: L^2(\Omega)^m\to L^2(\Omega)^m\qquad\text{setting}\qquad (\Pi_K
  \bsv)(x):=\Pi_K \bsv(x).
\end{align*}
It follows from  the convexity of $K$, with elementary computations,  that
$\Pi_K$ is $1$-Lipschitz and hence a generalized chain rule implies  $
\Pi_K:H^1(\Omega)^m\to H^1(\Omega)^m$ with
\begin{align}\label{eq:estPiK}
 |\nabla
  (\Pi_K \bsv)(x)|\le \operatorname{Lip}(\Pi_K) |\nabla \bsv(x)| = |\nabla
  \bsv(x)|\qquad x\in\Omega;
\end{align}
 compare with
\cite{BilFuc02,AmbrosioDalMaso:1990,AmbrosioFuscoPallara:2000}. 
This is the replacement for \eqref{eq:nablaU^*<nablaU}.

Let $\bsu\in H^1(\Omega)^m$ be such that 
\begin{align*}
  \int_\Omega \nabla\bsu\colon\nabla\bsv \dx = 0\qquad\text{for
    all}~\bsv\in H_0^1(\Omega)^m.
\end{align*}
 It is well known that $\bsu$ satisfies the
convex hull property
\begin{align*}
  \bsu(\Omega)\subset\operatorname{conv\,hull}\bsu(\partial\Omega)
\end{align*}
(see \cite{BilFuc02}) and that $\bsu$ is the unique minimizer of the energy 
\begin{align*}
  \energy(\bsv)\definedas \int_\Omega\frac12 |\nabla \bsv|^2\dx\qquad\text{in}~\bsu+ H_0^1(\Omega)^m.
\end{align*}
Let $\bsU\in H^1(\Omega)^m$ 
be some approximation of 
$\bsu$ and define
\begin{align*}
  \bsU^*\definedas \Pi_K\bsU, \qquad\text{with}~K\definedas
\operatorname{conv\,hull}\bsU(\partial\Omega).
\end{align*}
Consequently $\bsU^*$
satisfies the discrete convex hull property
\begin{align*}
  \bsU^*(\Omega)\subset\operatorname{conv\,hull}\bsU^*(\partial\Omega)
\end{align*}
and it follows from \eqref{eq:estPiK} similar as in the proof of Theorem
\ref{thm:main} that  $\bsU^*-\bsU\in
    H_0^1(\Omega)^m$ and 
\begin{align}
  \label{eq:mainv}
  \begin{split}
    \energy(\bsU^*)=\frac12\int_\Omega|\nabla \bsU^*|^2\dx\le
    \frac12\int_\Omega|\nabla
    \bsU|^2\dx=\energy(\bsU).
  \end{split}
\end{align}
Moreover, as in Corollary \ref{cor:main}, we obtain for
$\normm{\cdot}^2=\int_\Omega\abs{\nabla \cdot}^2\dx$ that
\begin{align}
  \label{eq:mainvb}
  \normm{\bsu-\bsU^*}\le \normm{\bsu-\bsU}\qquad\text{if}~\bsu-\bsU\in
  H_0^1(\Omega)^m. 
\end{align}
\begin{remark}\label{r:reactd}
  The results of this section can be generalized
  to reaction diffusion problems.  To this end we observe that, in
  order to prove  \eqref{eq:mainv},
  we need a replacement of \eqref{eq:U^*<U} for vector-valued
  functions. Note that for reaction diffusion problems the convex hull
  property reads as 
  \begin{align*}
    \bsu(\Omega)\subset\operatorname{conv\,hull}
    \big(\{0\}\cup\bsu(\partial\Omega)\big). 
  \end{align*}
  Therefore, defining $K\definedas
  \operatorname{conv\,hull}(\{0\}\cup\bsU(\partial\Omega))$ we obtain
  $|\bsU^*(x)|\le|\bsU(x)|$, $x\in\Omega$. This is the required replacement of
  \eqref{eq:U^*<U}.  
\end{remark}

\subsection{Nonlinear Problems}
\label{sec:nonlinear-problems}
Recalling the proof of Theorem \ref{thm:main} and \eqref{eq:U*props},
we observe that formally \eqref{eq:reduction} holds for energies 
\begin{align*}
  \energy(v)=\int_\Omega \mathcal{F}(x,|v|,|\nabla v|)\dx- F(v)
\end{align*}
with $\mathcal{F}:\Omega\times\setR\times\setR\to \setR$ being monotone
in its second and third variable. In particular,
Theorem~\ref{thm:main} can be applied to many nonlinear problems. 

As an example, we consider the nonlinear $p$-Laplace problem. To this
end, for fixed $p\in(1,\infty)$, $\frac1p+\frac1q=1$, let
$W^{1,p}(\Omega)$ 
be the space of $p$
integrable functions on $\Omega$ with $p$-integrable weak derivatives.
Let $u\in W^{1,p}(\Omega)$ such that   
\begin{alignat*}{2}
  -\divo |\nabla u|^{p-2}\nabla u &=: F\le0&\qquad&\text{in}~\Omega;
\end{alignat*}
in the distributional sense. 
Then $u$
is the unique minimizer of the energy 
\begin{align*}
  \energy(v)\definedas \int_\Omega \frac1p\abs{\nabla v}^p\dx
  -F(v)\qquad\text{in}~u+W^{1,p}_0(\Omega).
\end{align*}
Here $W^{1,p}_0(\Omega)$ is the subspace of
functions in $W^{1,p}(\Omega)$ with vanishing trace on $\partial\Omega$.
Moreover, it is well known that $u$ satisfies
\begin{align*}
  \sup_{\Omega}u\le \sup_{\partial\Omega}u;
\end{align*}
see
e.g. \cite{OttLM98}. Let $U\in W^{1,p}(\Omega)$ be some approximation
to $u$ and let 
\begin{align*}
  U^*\definedas \min\big\{U,\sup_{\partial\Omega}U\big\}.
\end{align*}
This implies \eqref{eq:nablaU^*<nablaU}. 
Note that the differential operator contains no reactive term and thus
we do not require \eqref{eq:U^*<U} to conclude, similar as in the proof of Theorem \ref{thm:main}, that 
\begin{align*}
  \energy(U^*)\le\energy(U)  \qquad\text{and thus}\qquad  \energy(U^*)-\energy(u)\le \energy(U)-\energy(u).
\end{align*}

In order to prove a result analog to Corollary \ref{cor:main} it remains to
correlate the energy difference to a reasonable measure of distance. 
This can be done using the so-called quasi-norm introduced by Barrett
and Liu in \cite{BarrettLiu:93}. It follows from
\cite[Lemma 13]{Kreuzer:12} that there exist constants $C,c>0$
such that
\begin{align*}
  c\big(\energy(v)-\energy(u)\big)&\le
  \normm{v-u}^2_{(\nabla u)}
  \definedas \int_\Omega (\abs{u}+\abs{v})^{p-2}\abs{u-v}^2\dx\le  C\big(\energy(v)-\energy(u)\big)
\end{align*}
for all $v\in u+W^{1,p}_0(\Omega)$. Hence,
with $v=U$ respectively $v=U^*$, it follows from above, that
\begin{align*}
  \sqrt{c}\,\normm{u-U^*}_{(\nabla u)}\le   \sqrt{C}\,\normm{u-U}_{(\nabla u)}.
\end{align*}

\section{Conclusion}
\label{sec:conclusion}
Maximum principles often reflect physical behavior of solutions and 
it is therefore
desirable that numerical approximations satisfy a maximum
principle as well. In this paper we 
enforce the maximum principle by performing a simple
cutoff to the approximation. We show that for many problems 
this truncation even improves the approximation. 
Hence, all error estimates and convergence results for the
approximation can directly be applied to the error of the truncated function. 

We emphasize that we do not specify the sort of
approximation beyond that it is conforming. 
Therefore, among others, all presented results apply to all conforming finite
element approximation including high order elements and conforming
$hp$-methods.  There is no restriction on the underlying
partition of $\Omega$ like non obtuseness of the triangulation and
the results even apply to partitions involving complicated
element geometries.

\providecommand{\bysame}{\leavevmode\hbox to3em{\hrulefill}\thinspace}
\providecommand{\MR}{\relax\ifhmode\unskip\space\fi MR }
\providecommand{\MRhref}[2]{%
  \href{http://www.ams.org/mathscinet-getitem?mr=#1}{#2}
}
\providecommand{\href}[2]{#2}



\end{document}